\documentclass[amstex,10pt,reqno]{amsart}
\usepackage{amsmath,amsfonts,amssymb,amsthm,enumerate,hyperref,multicol,cite}
\usepackage{graphicx}
\usepackage[english]{babel}
\newtheorem{theorem}{Theorem}
\newtheorem{lemma}{Lemma}

\numberwithin{equation}{section}
\newcommand\norm[1]{\left\lVert#1\right\rVert}
\usepackage[centering, bottom=1.0in,top=1.0in]{geometry}

\usepackage[symbol]{footmisc}

\usepackage{fancyhdr}
\begin{document}

\dedicatory{In memory of a lovely human-being and Approx. Theorist Prof. D. V. Pai}
\leftline{ \scriptsize \it  }
\title[]{Characterization of deferred type statistical convergence and P-summability method for operators: Applications to $q$-Lagrange-Hermite operator}
\maketitle
\begin{center}
{\bf Purshottam Narain Agrawal$^1$, Rahul Shukla$^2$ \footnote{Corresponding author.}, and  Behar Baxhaku $^3$}
\vskip0.2in
$^{1,2}$Department of Mathematics\\
Indian Institute of Technology Roorkee\\
Roorkee-247667, India\\
$^1$ email: pnappfma@gmail.com\\
$^2$ email: rshukla@ma.iitr.ac.in,\\
$^3$ Department of Mathematics\\
University of Prishtina\\
Prishtina, Kosovo\\
$^3$email: behar.baxhaku@uni-pr.edu
\end{center}

\begin{abstract}
The present work considers two important convergence techniques, namely deferred type statistical convergence and P-summability method in respect of positive linear operators. With regard to these techniques, we state and prove two general non-trivial Korovkin type approximation results for such operators. Further, we define an operator based on multivariate q-Lagrange-Hermite polynomials and exhibit the applicability of our general theorems.

\end{abstract}

\section{Introduction: Deferred Weighted A-statistical convergence and P-summability method}
In Pure and Applied Mathematics, the idea of convergence of a sequence in a given space $X$ provides a lot of insight and applications. Once a sequence is convergent, we are free to use the analogy into continuity of a function and hence it is easy to dive in the areas of Topological Spaces, Measure Theory, Functional Analysis, Numerical Analysis, and Mathematical Modeling etc., but what would be the situation if the sequence is not convergent, can we still say something about it? The answers are indeed affirmative and occur in the form of Summability Theory, particularly Statistical Convergence. In the past few decades, this notion of convergence has been deeply studied and generalized. The methods not only catch the fine details of some special class of sequences but also provide the impression of convergence to such sequences(non-convergent in usual sense). In the present article, we discuss two such methods: a generalized form of statistical convergence and P-summability method, and their applications to the theory of approximation. In this row, we first, briefly, introduce the concept of statistical convergence and its deferred type generalization. After that, we shall discuss the P-summability method.

\noindent
Let $M$ be a subset of the set of natural numbers $\mathbb{N}$ and for each $n\in \mathbb{N}$, we define
$M_n = \{m\in M: m\leq n \}.$ The natural density $d(M)$ of the set $M$ is defined by the limit(if exists) of the sequence $\big<\frac{|M_{n}|}{n}\big>$. More precisely,
$$d(M) = \lim_{n\to \infty}\frac{|M_{n}|}{n}.$$ Researchers have proposed multifarious ways to define the density  and these definitions are playing a pivotal role in the areas of Number Theory and Graph Theory as well(see \cite{NRai, NRai1, Pandey}). Any sequence $\big<x_{n}\big> $ is called \textit{statistically convergent} to $l$ if, for each $\epsilon >0$, the set $\{k\in \mathbb{N}:k\leq n ~\text{and}~|x_{k}-l|\geq \epsilon\}$ has density zero, i.e.
$$\lim_{n\to \infty}\frac{|\{k\in \mathbb{N}:k\leq n ~\text{and}~|x_{k}-l|\geq \epsilon\}|}{n} = 0.$$  This definition clearly indicates that every convergent sequence is always statistically convergent, while the converse need not to be true in general. For a counter example, we can opt the sequence $\big<x_n\big> = \begin{cases}
1, & \text{if $n=m^2,$ $m\in \mathbb{N}$ } ,\\
0, & \text{otherwise}
\end{cases}
$. Clearly, $\big<x_n\big>$ is not convergent in the usual sense while one can easily verify that it is statistically convergent to 0. Karakaya et al. \cite{Kara} derived the concept of weighted statistical convergence and the idea was later modified by Mursaleen et al. in \cite{Mur}.\newline
Let us assume that $\big<s_{k}\big>$ be a sequence such that $s_k \geq 0$ and
$$S_n = \sum_{k=1}^{n}s_k, \quad s_1 >0,$$
denotes its partial sum. The sequence $\big<x_n\big>$ is called \textit{weighted statistically convergent} to a number $l$ if, for any given $\epsilon > 0$, the following holds
$$\lim_{n\to \infty}\frac{|\{k\in \mathbb{N}:k\leq S_n ~\text{and}~s_k|x_{k}-l|\geq \epsilon\}|}{S_n} = 0.$$ If $X_1$ and $X_2$ are sequence spaces such that for every infinite matrix $A=(a_{n,k}):X_1 \to X_2$, we have the transformation $(Ax)_n = \sum_{k=1}^{\infty}a_{n,k}x_{k}$. Then, the matrix $A$ is called regular if $\underset{n\to \infty}{\lim} (Ax)_n= l$ whenever $\underset{k\to \infty}{\lim} x_k= l$. For a non-negative regular matrix $A=(a_{n,k})$, Freedman et al. \cite{Free} proposed the concept of $A$-statistical convergence. The sequence $\big<x_n\big>$ is called $A$-statistically convergent to a number $l$ (denoted by $stat_A- \underset{n \to \infty}{\lim}x_n = l$), if
$$\lim_{n \to \infty}\sum_{k:|x_k-l|\geq \epsilon}a_{n,k}=0, \quad ~ \text{for every}~\epsilon >0.$$
Very recently, Srivastava et al. \cite{HM} derived a more general idea of $A$-statistical convergence and coined the word \textit{deferred weighted $A$-statistical convergence}. Suppose $\big<b_n\big>$ and $\big<c_n\big>$ are the sequences of non-negative integers satisfying the regularity conditions $b_n < c_n; \underset{n\to \infty}{\lim}c_n=\infty$. Now, if we set
$$S_n = \sum_{m=b_n+1}^{c_n}s_m,$$ for any given sequence $\big<s_n\big>$ of non-negative real numbers and for any given sequence $\big<x_n\big>$, its respective deferred weighted mean by $\rho_n= \frac{1}{S_n}\sum_{m=b_n+1}^{c_n}s_mx_m$. Then, the sequence $\big<x_n\big>$ is called \textit{deferred weighted summable} (denoted by $c^{DWS}-\underset{n\to \infty}{\lim}x_n =l$) to $l$ if $\underset{n \to \infty}{\lim}\rho_n = l$. Further, we call $\big<x_n\big>$ to be \textit{deferred weighted A-summable} to (denoted by $c^{DWS}_A-\underset{n\to \infty}{\lim}x_n =l$) a number $l$ if
$$\lim_{n\to \infty}\frac{1}{S_n}\sum_{m=b_n+1}^{c_n}\sum_{k=1}^{\infty}s_ma_{m,k}x_k = l.$$ Let $c^{DWS}$ be the space of all deferred weighted summable sequences then an infinite matrix $A=(a_{n,k})$, is called deferred weighted regular matrix if
$$(Ax)_n=\sum_{k=1}^{\infty}a_{n,k}x_k\in c^{DWS} \quad \text{for every convergent sequence}~ x=(x_n),$$
with $$c^{DWS}-\underset{n \to \infty}{\lim}(Ax)_n = stat_A- \underset{n \to \infty}{\lim}x_n.$$ Now, for a non-negative deferred weighted regular matrix $A=(a_{n,k})$ and $K_{\epsilon} \subset \mathbb{N}=\{k\in\mathbb{N}:|x_k-l|\geq \epsilon\}$, a sequence $\big<x_n\big>$ is said to be \textit{deferred weighted $A$-statistically convergent} to $l$ (denoted by $stat_{A}^{DW}-\underset{n \to \infty}{\lim}x_n = l$) if, for each $\epsilon >0$, the deferred weighted $A$-density of $K_{\epsilon}$ i.e. $d^A_{DW}(K_{\epsilon})$ is zero. It means that
\begin{eqnarray*}
d^A_{DW}(K_{\epsilon})= \lim_{n\to \infty}\frac{1}{S_n}\sum_{m=b_n+1}^{c_n}\sum_{k\in K_{\epsilon}}s_ma_{m,k}= 0.
\end{eqnarray*}
We call the sequence $\big<x_n\big>$ to be \textit{deferred weighted A-statistically} convergent to the number $l$ with the rate $o(\gamma_n)$ (please see \cite{Duman1, Duman2}) if
\begin{eqnarray*}
\lim_{n\to \infty}\frac{1}{\gamma_n}\bigg\{\frac{1}{S_n}\sum_{m=b_n+1}^{c_n}\sum_{k\in K_{\epsilon}}s_ma_{m,k}\bigg\}&=& 0.
\end{eqnarray*}
In this case, we write $x_n-l = stat_{A}^{DW}-o(\gamma_n)$.

After having much introduction of the statistical convergence and its various types of generalizations, we shall now introduce the next equally important convergence technique.

\noindent
Suppose that $p(u)=\sum_{n=1}^{\infty}p_nu^{n-1}$ is a power series of non-negative coefficients with radius of convergence $R \in (0, \infty]$. Any sequence $\big<x_n\big>$ is called convergent to a number $l$ by means of  P-summability method (or P-summable/Power series summability)\cite{Bor} if
\begin{eqnarray}\label{e1.1}
\lim_{u \to R^-}\frac{1}{p(u)}\sum_{n=1}^{\infty}x_np_nu^{n-1} = l.
\end{eqnarray}
Additionally, the P-summability method is said to be regular \cite{Bos} iff
\begin{eqnarray}\label{*}
\lim_{u \to R^-}\frac{p_nu^{n-1}}{p(u)} = 0,~ \text{for all}~n \in \mathbb{N}.
\end{eqnarray}
The idea of convergence by the P-summability method is more general in comparison to the classical convergence. One can see this by taking an example of the sequence $\big<x_n\big> = \begin{cases}
1, & \text{n is even}\\
0, & \text{n is odd}
\end{cases}
.$ Clearly, it is not convergent in the usual sense but if we take $p_n = 1,~\forall ~n\in \mathbb{N}$, we have $p(u) = \frac{1}{1-u}$, $|u|<1$, which implies that $R = 1,$ and hence from (\ref{e1.1}) we have
\begin{eqnarray*}
\lim_{u \to 1^-}(1-u)\sum_{n=1}^{\infty}x_nu^{n-1} &=& \lim_{u \to 1 -}\frac{(1-u)}{u}\sum_{n=1}^{\infty}u^{2n} = \dfrac{1}{2},
\end{eqnarray*}
therefore $\big<x_n\big>$ converges to $\frac{1}{2}$ in the sense of P-summability method.\\
Due to this generality of convergence by power series method over the classical one, many authors have made significant contributions in this direction. The interested reader may refer to \cite{Kra, Sta, Tas1, Tas} etc.

\section{General Theorems}

Let $\big(\mathfrak{B}(I_a^b)$,$\norm{.}\big)$ be the normed space of all bounded functions on $I_a^b = [a,b]$ with the sup- norm and $\mathfrak{C}(I_a^b)$ denote the subspace of $\mathfrak{B}(I_a^b)$ of all continuous functions on $I_a^b$. We consider
\begin{eqnarray*}
m &=& \sup_{x\in I_a^b} |x|.
\end{eqnarray*}
In our further consideration, we assume $\norm{\mathfrak{L_n}((s-x)^2)}=\mu_{n}^{(2)}$.\\
In the following theorem, we obtain the necessary and sufficient conditions for the deferred A-weighted statistical convergence of a sequence of positive linear operators.
\begin{theorem}\label{thm1}
Let $\mathfrak{L_n}: \mathfrak{C}(I_a^b) \to \mathfrak{B}(I_a^b)$ be a sequence of positive linear operators and $A=(a_{nk})$ be a non-negative deferred weighted regular matrix. Suppose $\big<a_n\big>$ and $\big<b_n\big>$ be sequences of non-negative integers satisfying the regularity conditions. If $g_i(s) = s^i, for~ i=0,1,2,$ then
\begin{eqnarray*}
stat_{A}^{DW}-\lim_{n}\norm{\mathfrak{L_n}(g_i)-g_i} = 0, \quad ~ \text{for}~ i=0,1,2;
\end{eqnarray*}
if and only if
\begin{eqnarray*}
stat_{A}^{DW}-\lim_{n}\norm{\mathfrak{L_n}(g)-g} = 0, \quad ~ \forall ~ g\in \mathfrak{C}(I_a^b).
\end{eqnarray*}
\end{theorem}
\begin{proof}
Let us assume that, for all $g\in \mathfrak{C}(I_a^b)$, we have
\begin{eqnarray*}
stat_{A}^{DW}-\lim_{n}\norm{\mathfrak{L_n}(g)-g} = 0,
\end{eqnarray*}
then
\begin{eqnarray*}
stat_{A}^{DW}-\lim_{n}\norm{\mathfrak{L_n}(g_i)-g_i} = 0,\; for\; i=0,1,2
\end{eqnarray*}
is obvious. Conversely, assume that
\begin{eqnarray*}
stat_{A}^{DW}-\lim_{n}\norm{\mathfrak{L_n}(g_i)-g_i} = 0,\; for\; i=0,1,2.
\end{eqnarray*}
Since $g\in \mathfrak{C}(I_a^b),$ we have
\begin{eqnarray*}
| g(s)-g(x)| &\leq & 2\norm{g}, \quad \forall ~ s,x\in I_a^b.
\end{eqnarray*}
As $g$ is uniformly continuous on $I_a^b$, for a given $\epsilon > 0,$ $\exists$ $\delta >0$ such that $|g(s)-g(x)| < \epsilon$, whenever $|s-x| < \delta$, $s,x\in I_a^b$. If $|s-x| \geq \delta,$ then
\begin{eqnarray*}
|g(s)-g(x)| &\leq & 2\norm{g} \frac{(s-x)^2}{\delta^2}.
\end{eqnarray*}
Hence, we can write
\begin{eqnarray}\label{1}
|g(s)-g(x)| &< & \epsilon + 2\norm{g} \frac{(s-x)^2}{\delta^2} \quad \forall ~ s,x \in I_a^b.
\end{eqnarray}
In view of the monotonicity of $\mathfrak{L_n}$, we get
\begin{eqnarray*}
|\mathfrak{L_n}(g(s);x)-g(x)| \leq \mathfrak{L_n}(|g(s)-g(x)|;x) + |g(x)||\mathfrak{L_n}(1;x)-1|,
\end{eqnarray*}
hence from (\ref{1}), we obtain
\begin{eqnarray*}
|\mathfrak{L_n}(g(s);x)-g(x)| \leq \mathfrak{L_n}\big(\epsilon + 2\norm{g} \frac{(s-x)^2}{\delta^2};x\big) + |g(x)||\mathfrak{L_n}(1;x)-1|.
\end{eqnarray*}
As, we see that
\begin{eqnarray*}
\mathfrak{L_n}(\epsilon ;x) +\frac{2\norm{g}}{\delta^2}  \mathfrak{L_n}\big((s-x)^2;x\big) &=& \epsilon + \epsilon \big( \mathfrak{L_n}(1;x)-1\big) + \frac{2\norm{g}}{\delta^2} \bigg[\big(\mathfrak{L_n}(s^2;x)-x^2\big)\nonumber\\
&&-2x\big(\mathfrak{L_n}(s;x)-x\big)+x^2\big(\mathfrak{L_n}(1;x)-1\big)\bigg],
\end{eqnarray*}
hence
\begin{eqnarray}\label{5}
\norm{\mathfrak{L_n}(g)-g} &\leq & \epsilon + \bigg(\norm{g}+\frac{2\norm{g}m^2}{\delta^2}+\epsilon\bigg)\norm{\mathfrak{L_n}(1)-1} + \frac{2\norm{g}}{\delta^2}\norm{\mathfrak{L_n}(s^2)-x^2}+\frac{4\norm{g}m}{\delta^2}\norm{\mathfrak{L_n}(s)-x}.\nonumber\\
{}
\end{eqnarray}
Now, for any $\epsilon^{'} > 0$, we define the followings sets:
\begin{eqnarray*}
U_1 &=& \{n\in \mathbb{N}: \norm{\mathfrak{L_n}(g)-g} \geq \epsilon^{'}\};\\
U_2 &=& \{n\in \mathbb{N}: \bigg(\norm{g}+\frac{2\norm{g}m^2}{\delta^2}+\epsilon\bigg)\norm{\mathfrak{L_n}(1)-1}\geq \frac{\epsilon^{'}-\epsilon}{3} \};\\
U_3 &=& \{n\in \mathbb{N}: \frac{2\norm{g}}{\delta^2}\norm{\mathfrak{L_n}(s^2)-x^2}\geq \frac{\epsilon^{'}-\epsilon}{3}\};\\
U_4 &=& \{n \in \mathbb{N}: \frac{4\norm{g}m}{\delta^2}\norm{\mathfrak{L_n}(s)-x} \geq \frac{\epsilon^{'}-\epsilon}{3}\},
\end{eqnarray*}
then it is obvious that $U_1\subset U_2 \cup U_3 \cup U_4$ and hence
\begin{eqnarray*}
\frac{1}{S_n}\sum_{m=b_n+1}^{c_n}\sum_{k\in U_1}s_m a_{m,k} &\leq & \frac{1}{S_n}\sum_{m=b_n+1}^{c_n}\sum_{k\in U_2}s_m a_{m,k} + \frac{1}{S_n}\sum_{m=b_n+1}^{c_n}\sum_{k\in U_3}s_m a_{m,k}+ \frac{1}{S_n}\sum_{m=b_n+1}^{c_n}\sum_{k\in U_4}s_m a_{m,k}.
\end{eqnarray*}
In view of our hypothesis, we have
\begin{eqnarray*}
\lim_{n \to \infty} \frac{1}{S_n}\sum_{m=b_n+1}^{c_n}\sum_{k\in U_i}s_m a_{m,k} = 0,
\end{eqnarray*}
for $i = 2,3$ and $4$. Therefore
\begin{eqnarray*}
\lim_{n \to \infty} \frac{1}{S_n}\sum_{m=b_n+1}^{c_n}\sum_{k\in U_1}s_m a_{m,k} =  stat_{A}^{DW}-\lim_{n}\norm{L_n(g)-g} = 0.
\end{eqnarray*}
This completes the proof of the theorem.
\end{proof}
Our next theorem determines the rate of deferred weighted A-statistical convergence by the operators $\mathfrak{L_n}(g)~ to~ g$, for all $g\in \mathfrak{C}(I_a^b)$. We recall the definition of modulus of continuity. For any continuous function $g: I_a^b \to \mathbb{R}$ and a given $\delta > 0$, the modulus of continuity $\omega(g;\delta)$ is defined as
\begin{eqnarray*}
w(g;\delta) &:=& \sup_{|s-x|\leq \delta}\{|g(s)-g(x)|: s,x \in I_a^b\}.
\end{eqnarray*} 
\begin{theorem}
Let the matrix $A$ and the sequences $\big<a_n\big>$ , $\big<b_n\big>$ be as defined in Theorem \ref{thm1}. Further, let  $\big<\gamma_{n}\big>$ be a positive non-increasing sequence of real numbers. For $g\in \mathfrak{C}(I_a^b)$, if
\begin{eqnarray}\label{e2.3}
\{\norm{g}+\omega(g;\sqrt{\mu_{n}^{(2)}})\}\norm{\mathfrak{L_n}(1)-1} + 2~ \omega(g;\sqrt{\mu_{n}^{(2)}}) &=& stat_{A}^{DW}- o(\gamma_n),~as~n\rightarrow\infty,
\end{eqnarray}
then
\begin{eqnarray*}
\norm{\mathfrak{L_n}(g)-g} & = & stat_{A}^{DW}-o(\gamma_n),~as~n\rightarrow\infty.
\end{eqnarray*}
\end{theorem}
\begin{proof}
For any $g \in \mathfrak{C}(I_a^b)$, it is well known that
\begin{eqnarray*}
|g(s)-g(x)| & \leq & \big{\lbrace}1+\frac{(s-x)^2}{\delta^2}\big{\rbrace}\omega(g;\delta) \quad \delta >0.
\end{eqnarray*}
Hence, by the monotonocity and linearity of operators $\mathfrak{L_n}$, we get
\begin{eqnarray*}
|\mathfrak{L_n}(g;x)-g(x)| &\leq & |g(x)||\mathfrak{L_n}(1;x)-1| + \omega(g;\delta) + \omega(g;\delta)|\mathfrak{L_n}(1;x)-1|\\
&&+~ \frac{\omega(g;\delta)}{\delta^2} \mathfrak{L_n}((s-x)^2;x),
\end{eqnarray*}
which implies that
\begin{eqnarray}\label{24}
\norm{\mathfrak{L_n}(g)-g} &\leq & \{\norm{g}+\omega(g;\delta)\}\norm{\mathfrak{L_n}(1)-1} + \omega(g;\delta) + \frac{\omega(g;\delta)}{\delta^2}\norm{\mathfrak{L_n}((s-x)^2)}.
\end{eqnarray}
Taking $\delta = \sqrt{\norm{\mathfrak{L_n}((s-x)^2)}}=\sqrt{\mu_{n}^{(2)}}$, we get
\begin{eqnarray*}
\norm{\mathfrak{L_n}(g)-g} &\leq & \{\norm{g}+\omega(g;\sqrt{\mu_{n}^{(2)}})\}\norm{\mathfrak{L_n}(1)-1} + 2~ \omega(g;\sqrt{\mu_{n}^{(2)}}).
\end{eqnarray*}
Taking into account our hypothesis (\ref{e2.3}), the proof of the theorem is completed.
\end{proof}
In the following result we obtain a Korovkin type approximation theorem for the operators $\mathfrak{L_n}$ using P- summability method.
\begin{theorem}\label{thm3}
For $g\in \mathfrak{C}(I_a^b)$, the sequence $\big<\mathfrak{L_n}\big>$ satisfies
\begin{eqnarray}\label{6}
\lim_{u \to R^-}\frac{1}{p(u)}\sum_{n=1}^{\infty}\norm{\mathfrak{L_n}(g)-g}p_nu^{n-1} = 0,
\end{eqnarray}
iff
\begin{eqnarray}\label{7}
\lim_{u \to R^-}\frac{1}{p(u)}\sum_{n=1}^{\infty}\norm{\mathfrak{L_n}(s^i)-x^i}p_nu^{n-1} = 0,
\end{eqnarray}
for i = 0,1,2.
\end{theorem}
\begin{proof}
It is easy to see that the equation $(\ref{6}) \implies (\ref{7})$, thus we shall prove the converse. Assume that the condition (\ref{7}) is true. Using (\ref{5}), we can write
\begin{eqnarray*}
\lim_{u \to R^-}\frac{1}{p(u)}\sum_{n=1}^{\infty}\norm{\mathfrak{L_n}(g)-g}p_nu^{n-1} &\leq & \epsilon + \bigg(\norm{g}+\frac{2\norm{g}m^2}{\delta^2}+\epsilon\bigg) \lim_{u \to R^-}\frac{1}{p(u)}\sum_{n=1}^{\infty}\norm{\mathfrak{L_n}(1)-1}c_nu^{n-1}\\
&&+~ \frac{2\norm{g}}{\delta^2}\lim_{u \to R^-}\frac{1}{p(u)}\sum_{n=1}^{\infty}\norm{\mathfrak{L_n}(s^2)-x^2}p_nu^{n-1}\\
&&+~ \frac{4\norm{g}m}{\delta^2}\lim_{u \to R^-}\frac{1}{p(u)}\sum_{n=1}^{\infty}\norm{\mathfrak{L_n}(s)-x}c_nu^{n-1}.
\end{eqnarray*}
Hence, in view of the hypothesis and the arbitrariness of $\epsilon$, we have the desired result.
\end{proof}
  
\begin{theorem}
If $g\in \mathfrak{C}(I_a^b)$ and $\phi(u)$ is some positive function on $(0,R)$ such that
\begin{eqnarray*}
\frac{1}{p(u)}\sum_{n=1}^{\infty}\{\norm{g}+\omega(g;\sqrt{\mu_{n}^{(2)}})\}\norm{\mathfrak{L_n}(1)-1}p_nu^{n-1}+\frac{2}{p(u)}\sum_{n=1}^{\infty}\omega(g;\sqrt{\mu_{n}^{(2)}}) p_nu^{n-1} = O(\phi(u)), \quad \text{as $u\to R^{-}$}.
\end{eqnarray*}
Then,
\begin{eqnarray*}
\frac{1}{p(u)}\sum_{n=1}^{\infty}\norm{\mathfrak{L_n}(g)-g} p_nu^{n-1} = O(\phi(u)), \quad \text{as $u\to R^{-}$}.
\end{eqnarray*}
\end{theorem}
\begin{proof}
From the inequality (\ref{24}), we see that
\begin{eqnarray*}
\frac{1}{p(u)}\sum_{n=1}^{\infty}\norm{\mathfrak{L_n}(g)-g} p_nu^{n-1} &\leq & \frac{1}{p(u)}\sum_{n=1}^{\infty}\{\norm{g}+\omega(g;\delta)\}\norm{\mathfrak{L_n}(1)-1}p_nu^{n-1} \\
&&+~ \frac{1}{p(u)}\sum_{n=1}^{\infty}\omega(g;\delta)\bigg( 1+ \frac{1}{\delta^2}\mu_{n}^{(2)}\bigg)p_nu^{n-1}.
\end{eqnarray*}
Now, choosing $\delta=\sqrt{\mu_{n}^{(2)}}$ and considering the hypothesis, we reach the required result.
\end{proof}

\section{Construction of the $q$-Lagrange-Hermite operators}
The past few decades have been witness to the multivariate extension of many well-known orthogonal polynomials, some notably mentioned works may be refer to \cite{Altin3, Ekta, Altin}. Very recently, researchers have started to made efforts to construct sequence of linear positive operators (LPOs) using the multivariate polynomials. In this row, firstly, for $g \in \mathfrak{C}(I_0^1)$, Erku\c{s} et al. \cite{Erkus} have defined a sequence of LPO using multivariate Lagrange polynomials and studied its approximation behavior by means statistical convergence method.  Subsequently, Erku\c{s}-Duman et al. \cite{Erkus-Dum1} proposed and studied a Kantorovich version of the LPO. Furthermore, using the multivariate $q$-Lagrange polynomials defined in \cite{Altin}, Erku\c{s}-Duman \cite{Erkus-Dum2} constructed a $q$-(basic)analogue of the operator defined in \cite{Erkus}. With time, unknowingly, Mursaleen et al. \cite{Mursalen} independently attempted to define a $q$-(basic) analogue of the operator defined in \cite{Erkus}, while Behar et al. \cite{Bax} proposed a slight modification in the attempted operator and extended the study to the bi-variate and GBS (Generalized Boolean Sum) cases.

In order to explain the construction of $q$-Lagrange-Hermite LPO, we shall need to recall some important definitions from the quantum calculus (widely known as q-calculus). Let us assume that $|q| <1$ then we recall the $q$-(basic) analogues of a natural number and Pochhammer symbol (please see \cite{Treatise}) as
\begin{eqnarray*}
[n]_q := \frac{1-q^n}{1-q}; \quad \text{and} \quad (\rho;q)_n=
\begin{array}{cc}
  \bigg\{\begin{array}{cc}
  1 ,& \text{if} \quad n=0 ,\\
  (1-\rho)(1-\rho q)...(1-\rho q^{n-1}),& \text{if} \quad n\in \mathbb{N}. \\
    \end{array}
\end{array},
\end{eqnarray*}
respectively, where $\rho$ is some arbitrary parameter.

Inspired by the multivariate extension of Lagrange polynomials \cite{Chan}, Altin et al. \cite{Altin2} introduced the multivariate Lagrange-Hermite polynomials $h_p^{(\beta^{(1)},\beta^{(2)},\dotsb~,\beta^{(r)})}(z_1, z_2,...z_n)$ generated by the expression
\begin{eqnarray*}
\prod\limits_{i=1}^{r}\left(1-t^i z_i\right)^{-\beta^{(i)}}=\sum\limits_{p=0}^{\infty}h_p^{(\beta^{(1)},\beta^{(2)},\dotsb~\beta^{(r)})}(z_1,z_2,\dotsb~z_r)t^p,
\end{eqnarray*}
with $|t|<\min\{|z_1|^{-1},|z_1|^{-1/2},\dotsb~|z_r|^{-1/r}\}$. Recently, Erkus \cite{Erkus2} proposed a $q$-(basic) analogue of the above Lagrange-Hermite polynomial generated by
\begin{eqnarray}\label{8}
\prod\limits_{i=1}^{r}\frac{1}{\big(z_it^i;q\big)_{\beta^{(i)}}} = \sum\limits_{p=0}^{\infty}h_{p,q}^{(\beta^{(1)},\beta^{(2)},\dotsb~\beta^{(r)})}(z_1,z_2,\dotsb~z_r)t^p.
\end{eqnarray}
From (\ref{8}), it is easy to obtain the explicit form of the polynomial $h_{p,q}^{(\beta^{(1)},\beta^{(2)},\dotsb~\beta^{(r)})}(z_1,z_2,\dotsb~z_r)$ as
\begin{eqnarray*}\label{9}
h_{p,q}^{(\beta^{(1)},\beta^{(2)},\dotsb~\beta^{(r)})}(z_1,z_2,\dotsb~z_r) = \sum\limits_{l_1+2l_2+\dotsb+rl_r=p}\left\{\prod\limits_{k=1}^{r}\left(q^{\beta^{(k)}},q\right)_{l_k}\frac{(z_k)^{l_k}}{\left(q,q\right)_{l_k}}\right\}.
\end{eqnarray*}
Hence, using (\ref{8}), we have the identity
\begin{eqnarray}\label{10}
\prod\limits_{i=1}^{r}\big(z_it^i;q\big)_{\beta^{(i)}}\sum\limits_{p=0}^{\infty}\bigg(\sum\limits_{l_1+2l_2+\dotsb+rl_r=p}\left\{\prod\limits_{k=1}^{r}\left(q^{\beta^{(k)}},q\right)_{l_k}\frac{(z_k)^{l_k}}{\left(q,q\right)_{l_k}}\right\}\bigg)t^p = 1.
\end{eqnarray}
Now motivated by the above studies, for $g \in \mathfrak{C}(I_0^1)$, we propose the following sequence of Lagrange-Hermite type LPO:
\begin{eqnarray}\label{11}
\mathfrak{R}_{n,q}^{\alpha^{(1)},\dotsb,\alpha^{(r)}}(g;x) &=& \left\{\prod\limits_{i=1}^{r}\big(\alpha_i x^i;q\big)_{n}\right\}
\sum\limits_{p=0}^{\infty}\bigg\{\sum\limits_{l_1+2l_2+\dotsb+rl_r=p}\left\{\prod\limits_{k=1}^{r}\left(q^{n},q\right)_{l_k}\frac{(\alpha_{n}^{(k)})^{l_k}}{\left(q,q\right)_{l_k}}\right\}\nonumber\\
&&g\left(\frac{[l_{1}]_{q}}{[n+l_{1}-1]_{q}}\right)\bigg\}x^p,
\end{eqnarray}
where $\alpha^{(j)}=\big<\alpha_n^{(j)}\in (0,1)\big>_{n\in \mathbb{N}} $ are sequences of real numbers.\\
In the following lemmas, we obtain estimates for some raw moments of the operators defined by (\ref{11}).

\begin{lemma}\label{lem1}
The operators $\mathfrak{R}_{n,q}^{\alpha^{(1)},\dotsb,\alpha^{(r)}}(.;x)$ satisfy
\begin{eqnarray*}
\mathfrak{R}_{n,q}^{\alpha^{(1)},\dotsb,\alpha^{(r)}}(1;x)=1, \quad \text{for all}~ x\in I_0^1.
\end{eqnarray*}
\end{lemma}
\begin{proof}
Using (\ref{10}), proof of the lemma is straight-forward hence the details are omitted.
\end{proof}

\begin{lemma}\label{lem2}
For the operators $\mathfrak{R}_{n,q}^{\alpha^{(1)},\dotsb,\alpha^{(r)}}(.;x)$, we have
\begin{eqnarray*}
\mathfrak{R}_{n,q}^{\alpha^{(1)},\dotsb,\alpha^{(r)}}(s;x)=x \alpha_{n}^{(1)}.
\end{eqnarray*}
\end{lemma}
\begin{proof}
From the definition (\ref{11}) of the Lagrange-Hermite operator, we can write
\begin{eqnarray*}
\mathfrak{R}_{n,q}^{\alpha^{(1)},\dotsb,\alpha^{(r)}}(s;x) &=& \left\{\prod\limits_{i=1}^{r}\big(\alpha_i x^i;q\big)_{n}\right\}
\sum\limits_{p=1}^{\infty}\bigg\{\underset{l_r\geq 1}{\sum\limits_{l_1+2l_2+\dotsb+rl_r=p}}\left\{\prod\limits_{k=1}^{r}\left(q^{n},q\right)_{l_k}\frac{(\alpha_{n}^{(k)})^{l_k}}{\left(q,q\right)_{l_k}}\right\}\nonumber\\
&&\left(\frac{[l_{1}]_{q}}{[n+l_{1}-1]_{q}}\right)\bigg\}x^p.
\end{eqnarray*}
Using some elementary transformations, $\frac{[l_1]_{q}}{(q;q)_{l_1}}=\frac{1}{(1-q)(q;q)_{l_1-1}}$ and $\frac{(q^{n};q)_{l_1}}{[n+l_1-1]_{q}}=(1-q)(q^{n};q)_{l_1-1}$, we have
\begin{eqnarray*}
\mathfrak{R}_{n,q}^{\alpha^{(1)},\dotsb,\alpha^{(r)}}(s;x) &=& x \alpha_{n}^{(1)} \left\{\prod\limits_{i=1}^{r}\big(\alpha_i x^i;q\big)_{n}\right\}
\sum\limits_{p=1}^{\infty}\bigg\{\underset{l_r\geq 1}{\sum\limits_{l_1+2l_2+\dotsb+rl_r-1=p-1}}(q^{n};q)_{l_1-1}\dotsb(q^{n};q)_{l_r}\\
&& \frac{(\alpha_n^{(1)})^{l_1-1}\dotsb(\alpha_n^{(r)})^{l_r}}{(q;q)_{l_1-1}\dotsb(q;q)_{l_r}}\bigg\}x^{p-1}\\
&=& x \alpha_{n}^{(1)} \left\{\prod\limits_{i=1}^{r}\big(\alpha_i x^i;q\big)_{n}\right\}\sum\limits_{p=1}^{\infty}h^{(n,\dotsb,n)}_{p-1,q}(\alpha_n^{(1)},\alpha_n^{(2)},\dotsb, \alpha_n^{(r)})x^{p-1}\\
&=&x \alpha_{n}^{(1)} \left\{\prod\limits_{i=1}^{r}\big(\alpha_i x^i;q\big)_{n}\right\}\sum\limits_{p=0}^{\infty}h^{(n,\dotsb,n)}_{p,q}(\alpha_n^{(1)},\alpha_n^{(2)},\dotsb, \alpha_n^{(r)})x^{p}\\
&=& x \alpha_{n}^{(1)}, \quad  \text{in view of (\ref{10})}.
\end{eqnarray*}
\end{proof}

\begin{lemma}\label{lem3} The operators $\mathfrak{R}_{n,q}^{\alpha^{(1)},\dotsb,\alpha^{(r)}}(.;x)$ satisfy
$$\mathfrak{R}_{n,q}^{\alpha^{(1)},\dotsb,\alpha^{(r)}}(s^{2};x) \leq q(x\alpha_{n}^{(1)})^{2} + \frac{x\alpha_{n}^{(1)}}{[n]_q};$$ and $$\big|\mathfrak{R}_{n,q}^{\alpha^{(1)},\dotsb,\alpha^{(r)}}(s^{2}-x^{2};x)\big| \leq 2x^{2}(1-\alpha_{n}^{(1)})+\frac{x\alpha_{n}^{(1)}}{[n]_q}.$$
\end{lemma}
\begin{proof}
From the definition of the operator $\mathfrak{R}_{n,q}^{\alpha^{(1)},\dotsb,\alpha^{(r)}}(.;x)$, we have
\begin{eqnarray}\label{e3.4}
\mathfrak{R}_{n,q}^{\alpha^{(1)},\dotsb,\alpha^{(r)}}(s^2;x) &=& \left\{\prod\limits_{i=1}^{r}\big(\alpha_i x^i;q\big)_{n}\right\}
\sum\limits_{p=1}^{\infty}\bigg\{\underset{l_r\geq 1}{\sum\limits_{l_1+2l_2+\dotsb+rl_r=p}}\left\{\prod\limits_{k=1}^{r}\left(q^{n},q\right)_{l_k}\frac{(\alpha_{n}^{(k)})^{l_k}}{\left(q,q\right)_{l_k}}\right\}\left(\frac{[l_{1}]_{q}}{[n+l_{1}-1]_{q}}\right)^2\bigg\}x^p\nonumber\\
&=& x \alpha_{n}^{(1)} \left\{\prod\limits_{i=1}^{r}\big(\alpha_i x^i;q\big)_{n}\right\}
\sum\limits_{p=1}^{\infty}\bigg\{\underset{l_r\geq 1}{\sum\limits_{l_1+2l_2+\dotsb+rl_r-1=p-1}}(q^{n};q)_{l_1-1}\dotsb(q^{n};q)_{l_r}\nonumber\\
&&\frac{[l_{1}]_{q}}{[n+l_{1}-1]_{q}} \frac{(\alpha_n^{(1)})^{l_1-1}\dotsb(\alpha_n^{(r)})^{l_r}}{(q;q)_{l_1-1}\dotsb(q;q)_{l_r}}\bigg\}x^{p-1}\nonumber\\
&=& x \alpha_{n}^{(1)} \left\{\prod\limits_{i=1}^{r}\big(\alpha_i x^i;q\big)_{n}\right\}
\sum\limits_{p=1}^{\infty}\bigg\{\underset{l_r\geq 1}{\sum\limits_{l_1+2l_2+\dotsb+rl_r-1=p-1}}(q^{n};q)_{l_1-1}\dotsb(q^{n};q)_{l_r}\nonumber\\
&&\bigg(\frac{1+q[l_{1}-1]_{q}}{[n+l_{1}-1]_{q}}\bigg) \frac{(\alpha_n^{(1)})^{l_1-1}\dotsb(\alpha_n^{(r)})^{l_r}}{(q;q)_{l_1-1}\dotsb(q;q)_{l_r}}\bigg\}x^{p-1} = \sum_1 + \sum_2,\quad  \text{say.}
\end{eqnarray}
Here,
\begin{eqnarray*}
\sum_1 &=& x \alpha_{n}^{(1)} \left\{\prod\limits_{i=1}^{r}\big(\alpha_i x^i;q\big)_{n}\right\}
\sum\limits_{p=1}^{\infty}\bigg\{\underset{l_r\geq 1}{\sum\limits_{l_1+2l_2+\dotsb+rl_r-1=p-1}}(q^{n};q)_{l_1-1}\dotsb(q^{n};q)_{l_r}\nonumber\\
&&\bigg(\frac{1}{[n+l_{1}-1]_{q}}\bigg) \frac{(\alpha_n^{(1)})^{l_1-1}\dotsb(\alpha_n^{(r)})^{l_r}}{(q;q)_{l_1-1}\dotsb(q;q)_{l_r}}\bigg\}x^{p-1},
\end{eqnarray*}
since $\frac{1}{[n+l_{1}-1]_{q}} \leq \frac{1}{[n]_q}$, then using (\ref{10}), we have
\begin{eqnarray*}
\sum_1 &\leq & \frac{x \alpha_{n}^{(1)}}{[n]_q}.
\end{eqnarray*}
Now
\begin{eqnarray*}
\sum_2 &=& qx \alpha_{n}^{(1)} \left\{\prod\limits_{i=1}^{r}\big(\alpha_i x^i;q\big)_{n}\right\}
\sum\limits_{p=1}^{\infty}\bigg\{\underset{l_r\geq 1}{\sum\limits_{l_1+2l_2+\dotsb+rl_r-1=p-1}}(q^{n};q)_{l_1-1}\dotsb(q^{n};q)_{l_r}\nonumber\\
&&\bigg(\frac{[l_1-1]_q}{[n+l_{1}-1]_{q}}\bigg) \frac{(\alpha_n^{(1)})^{l_1-1}\dotsb(\alpha_n^{(r)})^{l_r}}{(q;q)_{l_1-1}\dotsb(q;q)_{l_r}}\bigg\}x^{p-1}\\
&=& q(x \alpha_{n}^{(1)})^2 \left\{\prod\limits_{i=1}^{r}\big(\alpha_i x^i;q\big)_{n}\right\}
\sum\limits_{p=2}^{\infty}\bigg\{\underset{l_r\geq 2}{\sum\limits_{l_1+2l_2+\dotsb+rl_r-2=p-2}}(q^{n};q)_{l_1-2}\dotsb(q^{n};q)_{l_r}\nonumber\\
&&\bigg(\frac{(1-q^{l_1-1})(1-q^{n+l_1-2})}{(1-q^{l_1-1})(1-q)[n+l_{1}-1]_{q}}\bigg) \frac{(\alpha_n^{(1)})^{l_1-2}\dotsb(\alpha_n^{(r)})^{l_r}}{(q;q)_{l_1-2}\dotsb(q;q)_{l_r}}\bigg\}x^{p-2}\\
&=& q(x \alpha_{n}^{(1)})^2\left\{\prod\limits_{i=1}^{r}\big(\alpha_i x^i;q\big)_{n}\right\}
\sum\limits_{p=2}^{\infty}\bigg\{\underset{l_r\geq 2}{\sum\limits_{l_1+2l_2+\dotsb+rl_r-2=p-2}}(q^{n};q)_{l_1-2}\dotsb(q^{n};q)_{l_r}\nonumber\\
&& \bigg(\frac{[n+l_1-2]_q}{[n+l_{1}-1]_{q}}\bigg)\frac{(\alpha_n^{(1)})^{l_1-2}\dotsb(\alpha_n^{(r)})^{l_r}}{(q;q)_{l_1-2}\dotsb(q;q)_{l_r}}\bigg\}x^{p-2},
\end{eqnarray*}
since $\frac{[n+l_1-2]_q}{[n+l_{1}-1]_{q}} < 1$,  in view of (\ref{10}) we get
\begin{eqnarray*}
\sum_2 & \leq &  q(x \alpha_{n}^{(1)})^2.
\end{eqnarray*}
Finally, using the estimates of $\sum_1$ and $\sum_2$ in (\ref{e3.4}), we obtain
\begin{eqnarray}\label{12}
\mathfrak{R}_{n,q}^{\alpha^{(1)},\dotsb,\alpha^{(r)}}(s^2;x) &\leq & q(x \alpha_{n}^{(1)})^2 + \frac{x \alpha_{n}^{(1)}}{[n]_q}.
\end{eqnarray}
Hence, we can write
\begin{eqnarray*}
\mathfrak{R}_{n,q}^{\alpha^{(1)},\dotsb,\alpha^{(r)}}(s^2;x) -x^2 &\leq & x^{2} (q(\alpha_{n}^{(1)})^{2}-1)+ \frac{x\alpha_{n}^{(1)}}{[n]_q}= -x^{2}(1-q(\alpha_{n}^{(1)})^{2})+\frac{x\alpha_{n}^{(1)}}{[n]_q}.
\end{eqnarray*}
As $q,\alpha_{n}^{(1)} \in (0,1)$, we have
\begin{eqnarray}\label{13}
\mathfrak{R}_{n,q}^{\alpha^{(1)},\dotsb,\alpha^{(r)}}(s^2;x) -x^2 &\leq & \frac{x\alpha_{n}^{(1)}}{[n]_q}.
\end{eqnarray}
Since the operator $\mathfrak{R}_{n,q}^{\alpha^{(1)},\dotsb,\alpha^{(r)}}(.;x)$ is a LPO, using Lemmas \ref{lem1} and \ref{lem2}, we may write
\begin{eqnarray*}
0 \leq \mathfrak{R}_{n,q}^{\alpha^{(1)},\dotsb,\alpha^{(r)}}((s-x)^2;x)&=& \mathfrak{R}_{n,q}^{\alpha^{(1)},\dotsb,\alpha^{(r)}}(s^2;x) - 2x \mathfrak{R}_{n,q}^{\alpha^{(1)},\dotsb,\alpha^{(r)}}(s;x) + x^2;\\
\implies \quad -2x^{2}(1-\alpha_{n}^{(1)}) &\leq & \mathfrak{R}_{n,q}^{\alpha^{(1)},\dotsb,\alpha^{(r)}}(s^{2};x)-x^{2}\\
\text{or,}\quad -2x^{2}(1-\alpha_{n}^{(1)})-\frac{x\alpha_{n}^{(1)}}{[n]_q} &\leq & \mathfrak{R}_{n,q}^{\alpha^{(1)},\dotsb,\alpha^{(r)}}(s^{2};x)-x^{2}.
\end{eqnarray*}
Hence, in view of (\ref{13}), we obtain
\begin{eqnarray*}
\big|\mathfrak{R}_{n,q}^{\alpha^{(1)},\dotsb,\alpha^{(r)}}(s^{2};x)-x^{2}\big| &\leq & 2x^{2}(1-\alpha_{n}^{(1)})+\frac{x\alpha_{n}^{(1)}}{[n]_q}.
\end{eqnarray*}
This complete the lemma.
\end{proof}

\section{Applications}
In this section, we show the convergence of our $q$-Lagrange-Hermite operators $\mathfrak{R}_{n,q}^{\alpha^{(1)},\dotsb,\alpha^{(r)}}$ by applying the earlier explained convergence techniques. For this purpose, we shall need the following important assumptions:\newline
\textit{\textbf{Assumption -1}} We consider the sequence $\alpha^{(1)}_{n} \in (0,1)$, convergent (in classical sense) to 1 as n tends to infinity.\\
\textit{\textbf{Assumption -2}} We assume that the sequences $\big<q_n\big>$ and $\big<q_n^{n}\big>$ converge (again in classical sense) to 1 and $a\in [0, 1)$, respectively, as n tends to infinity.

\subsection{Convergence of $\mathfrak{R}_{n,q_n}^{\alpha^{(1)},\dotsb,\alpha^{(r)}}$ using Deferred weighted-A Statistical method}
To prove the said convergence of our operator, we first recall the following important characteristics of a deferred weighted regular matrix given by Srivastava et al. \cite{HM}:
\begin{theorem}
Suppose $\big<b_n\big>$ and $\big<c_n\big>$ are the sequences of non-negative integers. Then, an infinite matrix $A=(a_{n,k})$ is said to be a deferred weighted regular matrix iff
\begin{eqnarray}\label{22}
\sup_{n}\sum_{k=1}^{\infty}\frac{1}{S_n}\bigg|\sum_{m=b_n +1}^{c_n}s_ma_{m,k}\bigg| &<& \infty;\nonumber\\
\lim_{n\to \infty}\frac{1}{S_n}\sum_{m=b_n +1}^{c_n}s_ma_{m,k} &=& 0, \quad \forall ~k \in \mathbb{N};\\
\lim_{n\to \infty}\frac{1}{S_n}\sum_{m=b_n +1}^{c_n}\sum_{k=1}^{\infty}s_ma_{m,k} &=& 1.\nonumber
\end{eqnarray}
\end{theorem}
In view of Theorem \ref{thm1}, it follows that the sequence of operators $\big<\mathfrak{R}_{n,q_n}^{\alpha^{(1)},\dotsb,\alpha^{(r)}}(g)\big>$ converges deferred weighted A-statistically to $g$ if the following holds:
\begin{eqnarray*}
stat_{A}^{DW}-\lim_{n}\norm{\mathfrak{R}_{n,q_n}^{\alpha^{(1)},\dotsb,\alpha^{(r)}}(s^i)-x^i} = 0, \quad ~ \text{for}~ i=0,1,2.
\end{eqnarray*}
Applying Lemma \ref{lem1}, the above condition is obvious for the case $i=0$. Now, for $i=1$, using Lemma \ref{lem2}, we can write
\begin{eqnarray}\label{21}
\norm{\mathfrak{R}_{n,q_n}^{\alpha^{(1)},\dotsb,\alpha^{(r)}}(s)-x} &\leq & (1-\alpha^{(1)}_{n}).
\end{eqnarray}
Let $\epsilon >0$ be an arbitrary real number. If we consider the sets $V_1 =\{n\in \mathbb{N}: \norm{\mathfrak{R}_{n,q}^{\alpha^{(1)},\dotsb,\alpha^{(r)}}(s)-x}\geq \epsilon\}$ and $V_2 = \{n\in \mathbb{N}:(1-\alpha^{(1)}_{n})\geq \epsilon \}$ then from (\ref{21}), we have $V_1 \subseteq V_2$ and hence
\begin{eqnarray*}
\frac{1}{S_n}\sum_{m=b_n+1}^{c_n}\sum_{k\in V_1}s_m a_{m,k} &\leq & \frac{1}{S_n}\sum_{m=b_n+1}^{c_n}\sum_{k\in V_2}s_m a_{m,k}\\
\text{or,}\quad \lim_{n \to \infty}\frac{1}{S_n}\sum_{m=b_n+1}^{c_n}\sum_{k\in V_1}s_m a_{m,k} &\leq & \lim_{n \to \infty}\frac{1}{S_n}\sum_{m=b_n+1}^{c_n}\sum_{k\in V_2}s_m a_{m,k}.
\end{eqnarray*}
From Assumption -1, there exists a positive integer $n_{0}(\epsilon)$ such that $ \lim_{n \to \infty}\frac{1}{S_n}\sum_{m=b_n+1}^{c_n}\sum_{k\in V_2}s_m a_{m,k} =   \lim_{n \to \infty}\frac{1}{S_n}\sum_{m=b_n+1}^{c_n}\sum_{k\leq n_{0}}s_m a_{m,k}$, thus
\begin{eqnarray*}
\lim_{n \to \infty}\frac{1}{S_n}\sum_{m=b_n+1}^{c_n}\sum_{k\in V_1}s_m a_{m,k} &\leq &\lim_{n \to \infty}\frac{1}{S_n}\sum_{m=b_n+1}^{c_n}\sum_{k\leq n_{0}}s_m a_{m,k}.
\end{eqnarray*}
Now, using the fact given in (\ref{22}) to the right hand-side of the above inequality, we have
\begin{eqnarray*}
\lim_{n \to \infty}\frac{1}{S_n}\sum_{m=b_n+1}^{c_n}\sum_{k\in V_1}s_m a_{m,k} & = & 0\\
\text{or} \quad stat_{A}^{DW}-\lim_{n}\norm{\mathfrak{R}_{n,q}^{\alpha^{(1)},\dotsb,\alpha^{(r)}}(s)-x} &=& 0.
\end{eqnarray*}
Now, for the case $i=2$, using Lemma \ref{lem3}, we obtain
\begin{eqnarray}\label{23}
\norm{\mathfrak{R}_{n,q_n}^{\alpha^{(1)},\dotsb,\alpha^{(r)}}(s^2)-x^2} &\leq & 2(1-\alpha^{(1)}_{n}) + \frac{\alpha^{(1)}_{n}}{[n]_{q_n}}.
\end{eqnarray}
Again, we construct the sets $W_1 =\{n\in \mathbb{N}: \norm{\mathfrak{R}_{n,q}^{\alpha^{(1)},\dotsb,\alpha^{(r)}}(s^2)-x^2}\geq \epsilon\}$, $W_2 = \{n\in \mathbb{N}: (1-\alpha^{(1)}_{n}) > \frac{\epsilon}{4}\}$, and $W_3 = \{n\in \mathbb{N}: \frac{\alpha^{(1)}_{n}}{[n]_{q_n}} > \frac{\epsilon}{2}\}$. From (\ref{23}), it is easy to see that $W_1 \subseteq W_2 \cup W_3$ and therefore
\begin{eqnarray*}
\lim_{n \to \infty}\frac{1}{S_n}\sum_{m=b_n+1}^{c_n}\sum_{k\in W_1}s_m a_{m,k} &\leq & \lim_{n \to \infty}\frac{1}{S_n}\sum_{m=b_n+1}^{c_n}\sum_{k\in W_2}s_m a_{m,k} + \lim_{n \to \infty}\frac{1}{S_n}\sum_{m=b_n+1}^{c_n}\sum_{k\in W_3}s_m a_{m,k}.
\end{eqnarray*}
Finally, following the previous logic along with the assumptions and the fact given in (\ref{22}), the required claim is established.

\subsection{Convergence of $\mathfrak{R}_{n,q_n}^{\alpha^{(1)},\dotsb,\alpha^{(r)}}$ using P-summability method}
In order to show the uniform convergence of $\big<\mathfrak{K}_{n,q_n}^{\beta^{(1)},\dotsb,\beta^{(r)}}(g)\big>$ to $g$ on $I_0^1$ by P-summability method, in view of Theorem \ref{thm3}, it is sufficient to establish that
\begin{eqnarray*}
\lim_{u \to R^-} \frac{1}{p(u)}\sum_{n = 1}^{\infty}\norm{\mathfrak{R}_{n,q_n}^{\alpha^{(1)},\dotsb,\alpha^{(r)}}(s^{i})-x^i}p_{n}u^{n-1} &=& 0, \quad \text{for}~ i =0,1,2.
\end{eqnarray*}
Using Lemma \ref{lem1}, it is easy to verify for the case when $i=0$, i.e. \\$\lim_{u \to R^-} \frac{1}{p(u)}\sum_{n= 1}^{\infty}\norm{\mathfrak{R}_{n,q_n}^{\alpha^{(1)},\dotsb,\alpha^{(r)}}(1)-1}p_{n}u^{n-1} = 0$. Now, using Lemma \ref{lem2}, we have
\begin{eqnarray}\label{19}
\frac{1}{p(u)}\sum_{n= 1}^{\infty}\norm{\mathfrak{R}_{n,q_n}^{\alpha^{(1)},\dotsb,\alpha^{(r)}}(s)-x}p_{n}u^{n-1} &\leq & \frac{1}{p(u)}\sum_{n= 1}^{\infty}\big(1-\alpha^{(1)}_{n}\big)p_{n}u^{n-1}.
\end{eqnarray}
From our Assumption-1, for a given $\epsilon >0$ there exists a positive integer $n_0(\epsilon)$ such that $|\alpha^{(1)}_{n}-1|<\frac{\epsilon}{2}$ for all $n > n_{0}(\epsilon)$, thus
\begin{eqnarray*}
\frac{1}{p(u)}\sum_{n = 1}^{\infty}\bigg((1-\alpha^{(1)}_{n})\bigg) p_{n}u^{n-1} &\leq & \frac{1}{p(u)}\sum_{n = 1}^{n_0}(1-\alpha^{(1)}_{n}) p_{n}u^{n-1} +\frac{\epsilon}{2p(u)}\sum_{n = n_0 +1}^{\infty} p_{n}u^{n-1}\\
& < & \frac{1}{p(u)}\sum_{n = 1}^{n_0}(1-\alpha^{(1)}_{n}) p_{n}u^{n-1} +\frac{\epsilon}{2p(u)}\sum_{n =1}^{\infty} p_{n}u^{n-1}.
\end{eqnarray*}
Since $\big<1-\alpha_{n}^{(1)}\big>$ is a bounded sequence, $\exists$ $M_1 > 0$ such that $M_1 = \underset{1\leq n\leq n_0}{\max}(1-\alpha_{n}^{(1)})$ therefore
\begin{eqnarray}\label{19A}
\frac{1}{p(u)}\sum_{n = 1}^{\infty}\bigg((1-\alpha^{(1)}_{n})\bigg) p_{n}u^{n-1} & < & \frac{M_1}{p(u)}\sum_{n = 1}^{n_0} p_{n}u^{n-1} +\frac{\epsilon}{2}.
\end{eqnarray}
Now, in view of the regularity condition given by (\ref{*}) there exists $\delta_{j}(\epsilon) > 0$ such that $\frac{p_{j}u^{j-1}}{p(u)} < \frac{\epsilon}{2M_1n_0}$ for all $R-\delta_{j}(\epsilon) < u < R$, and $j = 1,2,...n_0(\epsilon)$. Consider $\delta(\epsilon) = \min\big(\delta_{1}(\epsilon),\delta_{2}(\epsilon), ...,\delta_{n_0}(\epsilon)\big)$ then $\frac{p_{j}u^{j-1}}{p(u)} < \frac{\epsilon}{2M_1n_0}$ for all $R-\delta(\epsilon) < u < R$ and for every $n = 1,2,...n_0$. Hence from (\ref{19A}) we have
\begin{eqnarray*}
\frac{1}{p(u)}\sum_{n = 1}^{\infty}\bigg((1-\alpha^{(r)}_{n})\bigg) p_{n}u^{n-1} & < & M_1\frac{\epsilon ~n_0}{2M_1n_0}  + \frac{\epsilon}{2} = \epsilon.
\end{eqnarray*}
Consequently, from (\ref{19}) we have
\begin{eqnarray*}
 \frac{1}{p(u)}\sum_{n= 1}^{\infty}\norm{\mathfrak{R}_{n,q_n}^{\alpha^{(1)},\dotsb,\alpha^{(r)}}(s)-x}p_{n}u^{n-1} &<& \epsilon, \quad \forall ~ u\in(R-\delta , R).
\end{eqnarray*}
By a similar reasoning and using Lemma \ref{lem3}, we can show that
\begin{eqnarray*}\label{20}
\frac{1}{p(u)}\sum_{n = 1}^{\infty}\norm{\mathfrak{R}_{n,q}^{\alpha^{(1)},\dotsb,\alpha^{(r)}}(s^{2};x)-x^2} p_{n}u^{n-1} &\leq &\frac{1}{p(u)}\sum_{n = 1}^{\infty}\bigg(2(1-\alpha_{n}^{(1)})+\frac{\alpha_{n}^{(1)}}{[n]_{q_n}}\bigg)p_{n}u^{n-1}\nonumber\\
&<& \epsilon, \quad \text{for some $\theta(\epsilon) >0$}~ and~ \forall ~ u\in(R-\theta , R),
\end{eqnarray*}
in view of the fact that the sequence $\big<\frac{\alpha_{n}^{(1)}}{[n]_{q_n}}\big> \to 0,$ as $n \to \infty$ (keeping in mind the Assumption -2). This completes the requirement.

\section{A Concluding remark on P-summability}
In this section, we highlight the fact that the convergence of the $q$-Lagrange-Hermite operator $\mathfrak{R}_{n,q_n}^{\alpha^{(1)},\dotsb,\alpha^{(r)}}$ through the P-summability method is indeed more powerful in comparison (or a non-trivial type generalization) to the well-known Korovkin theorem. For instance, consider the sequence:
\begin{eqnarray}\label{remeq1}
\mathfrak{H}_{n,q_n}^{\alpha^{(1)},\dotsb,\alpha^{(r)}}(g;x) = (1+y_{n})\mathfrak{R}_{n,q_n}^{\alpha^{(1)},\dotsb,\alpha^{(r)}}(g;x),
\end{eqnarray}
where $\big<y_n\big>=
\begin{array}{cc}
  \bigg\{\begin{array}{cc}
  1 ,& \text{if $n=m^2,$ $m\in \mathbb{N}$ } ,\\
  0,& \text{otherwise} \\
    \end{array}
\end{array}
$. Now if we take $p_n= 1,
$ for all $n \in \mathbb{N}$, then the power series $p(u)= \sum_{n=1}^{\infty}p_{n}u^{n-1}= \frac{1}{1-u}, |u| < 1 $ implies that $R = 1$. Further, we see that
\begin{eqnarray*}
\frac{1}{p(u)}\sum_{n=1}^{\infty}y_{n} p_{n}u^{n-1} = \frac{(1-u)}{u}\sum_{m=1}^{\infty}u^{m^2}.
\end{eqnarray*}
Applying Cauchy root test, the series $\sum_{m=1}^{\infty}u^{m^2}$ is absolutely convergent in $|u|<1$, therefore it follows that
\begin{eqnarray}\label{remeq2}
\lim_{u \to 1^-}\frac{1}{p(u)}\sum_{n=1}^{\infty}y_{n} p_{n}u^{n-1} = \lim_{u \to 1^-} \frac{(1-u)}{u}\sum_{m=1}^{\infty}u^{m^2} = 0.
\end{eqnarray}
Therefore, the sequence $\big<y_n\big>$ converges to zero through P-summability method. From equation (\ref{remeq1}) and Lemma \ref{lem3}, we have
\begin{eqnarray*}
\norm{\mathfrak{H}_{n,q_n}^{\alpha^{(1)},\dotsb,\alpha^{(r)}}(s)-x} &\leq & (1-\alpha_{n}^{(1)}) +  y_{n}\big(\alpha_{n}^{(1)} \big).
\end{eqnarray*}
Since $\big<(1-\alpha_{n}^{(1)})\big>$ converges to 0 as $n \to \infty $, it will also P-summable to zero. Further, from Assumption-2, we see that the sequence $\big<\alpha_{n}^{(1)}\big>$ is bounded, hence in view of (\ref{remeq2}), we have
\begin{eqnarray*}
\lim_{u \to 1^-}\frac{1}{p(u)}\sum_{n=1}^{\infty}p_{n}u^{n-1}\norm{\mathfrak{H}_{n,q_n}^{\alpha^{(1)},\dotsb,\alpha^{(r)}}(s)-x} = 0.
\end{eqnarray*}
Next, using the definition (\ref{remeq1}) and Lemma \ref{lem3}, we obtain
\begin{eqnarray*}
\norm{\mathfrak{H}_{n,q_n}^{\alpha^{(1)},\dotsb,\alpha^{(r)}}(s^2)-x^2} &\leq & \bigg(2\big(1-\alpha_{n}^{(1)})+\frac{\alpha_{n}^{(1)}}{[n]_{q_n}}\bigg) + y_n \bigg(q_n(\alpha_{n}^{(1)})^2+\frac{\alpha_{n}^{(1)}}{[n]_{q_n}}\bigg).
\end{eqnarray*}
Again, we note that the sequence $\big<2\big(1-\alpha_{n}^{(1)})+\frac{\alpha_{n}^{(1)}}{[n]_{q_n}}\big>$    converges to 0 through the P-summability method and $q_n(\alpha_{n}^{(1)})^2+\frac{\alpha_{n}^{(1)}}{[n]_{q_n}} \leq 2$ for all $n\in \mathbb{N}$, therefore using (\ref{remeq2}), we can write
\begin{eqnarray*}
\lim_{u \to 1^-}\frac{1}{p(u)}\sum_{n=1}^{\infty}p_{n}u^{n-1}\norm{\mathfrak{H}_{n,q_n}^{\alpha^{(1)},\dotsb,\alpha^{(r)}}(s^2)-x^2} = 0.
\end{eqnarray*}
Finally, using Theorem \ref{thm3}, we can say that the sequence $\big<\mathfrak{H}_{n,q_n}^{\alpha^{(1)},\dotsb,\alpha^{(r)}}(g;x)\big>$ is P-summable to $g$ while it is imperative to notice that the Korovkin theorem does not hold for the operator $\mathfrak{H}_{n,q_n}^{\alpha^{(1)},\dotsb,\alpha^{(r)}}(.;x)$ defined in (\ref{remeq1}), as the sequence $\big<y_n\big>$ is not convergent. This proves our claim.

\section{Conclusions and Comments}
In this work, we visited two summability methods namely deferred type statistical convergence and P-summability,  and have shown the applications of these methods to the constructive theory of approximation. With these two impressions of convergence, we established two non-trivial Korovkin type convergence theorems for LPOs. Using the multivariate $q$-Lagrange-Hermite polynomials, we constructed a new class of LPOs,  proved some interesting inequalities concerning to the moments, and finally have shown the convergence through the said summability techniques. By an example, we have shown in Section-5 that our Theorem \ref{thm3} provides a great deal over the well-known Korovkin theorem. We understand that as there are many summabilty techniques reported in literature, one may provide a better approximation to our concern. In this connection, interested readers may look into the R-type convergence \cite{Jcon}, Ideal convergence \cite{Pdas} and references therein for future scopes and better understanding.

\begin{center}
\textbf{Acknowledgments}
\end{center}
The second author is heartily indebted to his loving grandfather Mr. Nageshswar Prasad Shukla for all the motivations and introduction to S. Ramanujan's legacy. He also gratefully acknowledges the financial support given to him by the Ministry of Education, Govt. of India to carry out the above work.
\begin{center}
\textbf{Declarations}
\end{center}
\textbf{Conflict of interest:-} Not applicable;\\
\textbf{Data availability:-} We assert that no data sets were generated or analyzed during the preparation of the manuscript;\\
\textbf{Code availability:-} Not applicable;\\
\textbf{Authors' contributions:-} All the authors have equally contributed to the conceptualization, framing and writing of the manuscript.


\begin{thebibliography}{9}

\bibitem{Ekta}
R. Akta\c{s}, E. E. Duman, {\it The Laguerre polynomials in the several variables}, Math. Slovaca, 63(3) (2013), 531--544.

\bibitem{Altin3}
A. Altin, R. Akta\c{s}, E. E. Duman, {\it On a multivariate extension for the extended Jacobi polynomials}, J. Math. Anal. Appl. 353 (2009), 121--133.


\bibitem{Altin2} A. Altin, E. Erku\c{s}, {\it On a multivariate extension of the Lagrange-Hermite polynomials}, Integral Transforms Spec. Funct. 17(4) (2006), 239--244.


\bibitem{Altin}
A. Altin, E. Erku\c{s}, F. Ta\c{s}delen, {\it The q-Lagrange polynomials in several variables}, Taiwan. J. Math. 5 (2006), 1131-1137.

\bibitem{Bax} B. Baxhaku, P. N. Agrawal, R. Shukla, {\it Bivariate positive linear operators constructed by means of $q$- Lagrange polynomials}, J. Math. Anal. Appl. 491 (2020), https://doi.org/10.1016/j.jmaa.2020.124337.

\bibitem{Bos} J. Boos, {\it Classical and Modern Methods in Summability}, Oxford University Press (2000).

\bibitem{Bor} D. Borwein, {\it On summability method based on power series}, Proc. Roy. Soc. Edinburgh 64 (1957), 342-349.




\bibitem{Chan} W.C. C. Chan, C.J. Chyan, H. M. Srivastava, {\it The Lagrange polynomials in several variables,} Integral Transforms Spec. Funct. 12
(2001), 139-148.

\bibitem{Jcon}
J. Connor, {\it R-type summability methods, Cauchy criteria, P-sets and statistical convergence}, Proc. Amer. Math. Soc. 115 (1992), 319--327.

\bibitem{Pdas}
P. Das, {\it Some further results on ideal convergence in topological spaces}, Topology Appl. 159 (2012), 2621--2626.

\bibitem{Duman2}
O. Duman, M. K. Khan, C. Orhan, {\it A-statistical convergence of approximating operators}, Math. Inequal. Appl. 6(4) (2003), 689--699.


\bibitem{Duman1}
O. Duman, C. Orhan, {\it Rates of A-statistical convergence of positive linear operators}, Appl. Math. Lett. 18 (2005), 1339--1344.


\bibitem{Erkus-Dum1}  E. Erku\c{s}-Duman, O. Duman, \textit{Integral-type generalization of operators obtained from certain multivariate polynomials}, Calcolo 45(1) (2008), 53--67.


\bibitem{Erkus2} E. Erku\c{s}, {\it The multivariable $q$-Lagrange-Hermite polynomials}, Internat. J. Comput. Numer. Anal. Appl. 6 (2004), 143--151.

\bibitem{Erkus-Dum2} E. Erku\c{s}-Duman, \textit{Statistical approximation by means of operators constructed by $q$-Lagrange polynomials}, J. Comput. Anal. Appl. 14(1) (2012), 67--77.



\bibitem{Erkus} E. Erku\c{s}, O. Duman, H. M. Srivastava, {\it Statistical approximation of certain positive linear operators constructed by means of the
Chan-Chyan-Srivastava polynomials}, Appl. Math. Comput. 182 (2006), 213-222.

\bibitem{Free} A. R. Freedman, J. J. Sember, {\it Densities and summability}, Pacific J. Math., 95 (1981), 293-305.


\bibitem{Kara} V. Karakaya, T. A. Chishti, {\it Weighted statistical convergence}, Iran. J. Sci. Technol. A. Sci. 33 (2009), 219-223.

\bibitem{Kra} W. Kratz, U. Stadtm\'uler, {\it Tauberian theorems for $J_p$-summability}, J. Math. Anal. Appl. 139 (1989), 362-371.



\bibitem{Mur} M. Mursaleen, V. Karakaya, M. Ert$\ddot{u}$rk, F. G$\ddot{u}$rsoy, {\it Weighted statistical convergence and its application to Korovkin type approximation theorem}, Appl. Math. Comput. 218 (2012), 9132-9137.


\bibitem{Mursalen} M. Mursaleen, A. Khan, H. M. Srivastava and K. S. Nisar, {\it Operators constructed by means of q-Lagrange polynomials and
A-statistical approximation,} Appl. Math. Comput. 219 (2013), 6911-6918.

\bibitem{NRai} R. K. Pandey, N. Rai, {\it Maximal Density of sets with missing differences and various coloring parameters of distances graphs}, Taiwan. J. Math. 24(6) (2020), 1383--1397.

\bibitem{NRai1} R. K. Pandey, N. Rai, {\it Density of sets with missing differences and applications}, Math. Slovaca 71(3) (2021), 595--614. 

\bibitem{Pandey} R. K. Pandey, A. Tripathi, {\it On the density of integral sets with missing differences from sets related to arithmetic progression}, J. Number Theory 131 (2011), 634-647.

\bibitem{HM} H. M. Srivastava, B. B. Jena, S. K. Paikray, U. K. Misra, {\it Deferred weighted $A$-statistical convergence based upon the $(p,q)$-Lagrange polynomials and its applications to approximation theorems}, J. Appl. Anal. 24 (2018), 1-16.

\bibitem{Treatise} H. M. Srivastava, H. L. Manocha, {\it A Treatise of Genration Functions}, Halsted Press (Ellis Horwood Limited Chichester), John Wiley and Sons, New Work, 1984.

\bibitem{Sta} U. Stadtmuller, A. Tali, {\it On certain families of generalized Norlund methods and power
series methods}, J. Math. Anal. Appl. 238 (1999), 44-66.


\bibitem{Tas1} E. Tas, \"{O}. G. Atlihan, {\it Korovkin type approximation theorems via power series method}, S\~{a}o Paulo J. Math. Sci. 13 (2019), 696--707.  

\bibitem{Tas} E. Tas, T. Yurdakadim, {\it Approximation by positive linear operators in modular
spaces by power series method}, Positivity 21 (2017), 1293-1306.

 

\end{thebibliography}
\end{document}